\documentclass[12pt]{amsart}
\usepackage{graphicx}

\usepackage{amsmath,amsthm,amsfonts,amssymb,verbatim}

\usepackage{amsmath,hyperref}
\usepackage{graphics}
\usepackage[all]{xy}
\usepackage{ulem}
\usepackage[utf8]{inputenc}

\usepackage{amsmath}
\usepackage{amscd}

\newcommand{\CC}{{\mathbb C}}
\newcommand{\NN}{{\mathbb N}}
\newcommand{\ZZ}{\mathbb Z}
\newcommand{\QQ}{{\mathbb Q}}

\def\A{{\mathcal A}}

\def\L{{\mathcal L}}

\def\S{{\mathcal S}}

\def\dist{{\rm dist}}
\numberwithin{equation}{section}

\DeclareMathOperator{\Norm}{{\rm Norm}}

\DeclareMathOperator{\Aut}{{\rm Aut}}

\DeclareMathOperator{\GL}{{\rm GL}}

\newtheorem{theo}{Theorem}
\newtheorem{proposition}[theo]{Proposition}
\newtheorem{coro}[theo]{Corollary}
\newtheorem{lemma}[theo]{Lemma}

\begin{document}

\title[Centralizers of Cantor systems]{Realization of big centralizers of minimal aperiodic actions on the  Cantor set.}
  
  \author[Cortez]{Mar\'{\i}a Isabel Cortez}%
  \thanks{The research of the first author was supported by proyecto Fondecyt Regular No. 1190538.}
\address{Departamento de Matem\'atica y Ciencia de la Computaci\'on, Universidad de Santiago de Chile}  \email{maria.cortez@usach.cl}

\author[Petite]{Samuel Petite}
  \thanks{This research was supported through the cooperation project MathAmSud DCS 38889 TM.}
\address{ Laboratoire Ami\'enois
de Math\'ematique Fondamentale et Appliqu\'ee, CNRS-UMR 7352, Universit\'{e} de Picardie Jules Verne}\email{samuel.petite@u-picardie.fr}  
  
  \date{Dated: \today}  
  
  \maketitle
  
\begin{abstract}
{\noindent  In this article we study the centralizer  of a minimal aperiodic action of a countable group on the Cantor set (an aperiodic minimal Cantor system). 
We show that  any countable  residually finite group is the subgroup of the centralizer of some minimal $\ZZ$ action on the Cantor set, and that any countable group is the subgroup of the normalizer  of a minimal aperiodic action of an abelian countable free group on the Cantor set. On the other hand we show that for any countable group $G$, the centralizer of any minimal aperiodic $G$-action on the Cantor set is a subgroup of the centralizer of a minimal $\ZZ$-action.  }
\end{abstract}
  
 \section{Introduction}
   An {\it automorphism}   of  the topological dynamical system $(X, T, \Gamma)$ given by the continuous action $T:\Gamma\times X\to X$ of a countable group $\Gamma$  on the  (compact) topological space  $X$, is a self-homeomorphism of $X$ commuting with each transformation  $T(g,\cdot)$.  A classical question in dynamics is to understand the algebraic properties of the group of all  the automorphisms, also called {\it centralizer}, of a prescribed dynamical system and their relationships with the dynamical properties of the system.  In this paper we focus on  the case when the space $X$ is a Cantor set. One of the reasons for this choice is the topology of the Cantor set does not restrict the algebraic properties  of the groups of homeomorphisms because  any countable group acts faithfully on the Cantor set. This situation is very different in other spaces like manifolds. For instance,  only extensions of orderable  group can act faithfully on the circle \cite{Witte}. For higher dimensional compact manifolds, the restrictions of diffeomorphisms groups  fall in the scope of R. J. Zimmer's conjectures.

 \noindent 
The dynamical properties determine important restrictions on the  groups of automorphisms  that can be realized.  
For instance, Hedlund observed that the automorphism group of the dynamical systems  given by an expansive action  is always countable. 
\noindent 
The subsets which are invariant by the dynamics  can also  restrict the automorphism group \cite{SchraudnerSalo}.  For instance, in the case of irreducible $\ZZ$-subshift  of finite type, every automorphism preserves the finite set of periodic points of a given period  and since the periodic points are dense, their restrictions separate the automorphisms implying that the group of automorphisms is residually finite \cite{BLR}. 
 To avoid such limitations, we will focus on dynamical systems which are minimal, i.e, systems without proper invariant closed subsets. 
In this case, none of the former restrictions appears: the centralizer can be uncountable as  in the odometer action, or not residually finite as it was shown in \cite{BLR}, where the authors construct a minimal subshift whose centralizer contains a group isomorphic to the rationals $\QQ$.   
However it appears from recent works that the centralizers of zero entropy minimal subshifts  are very limited \cite{BaakeRobertsYassawi, CovenQuasYassawi, CyrKra, CyrKra2, CyrKra3, CyrKra5,DonosoDurandMaassPetite,DonosoDurandMaassPetite17}.  
  
  \noindent 
Another motivation to focus our attention on the centralizer of a minimal Cantor system comes from the study of (topological) full groups.  From the work of Juschenko and Monod \cite{JuschenkoMonod}, it is  known that the topological full group of a $\ZZ$ Cantor minimal system is a countable amenable group. Thanks to this result together with   those shown in \cite{Matui} by Matui,  the commutator  of the topological full groups of minimal $\ZZ$-subshifts become the first known examples of infinite groups which are amenable, simple and finitely generated. On the other hand,
Giordano, Putnam, and Skau  \cite{GPS99}, and Medynets  \cite{Medynets} prove that abstract isomorphisms between full groups of Cantor minimal systems have a topological realization. More precisely, they show the outer  automorphism  group of a topological full group (of a  $\ZZ$ Cantor minimal system $(X,T, \ZZ)$) is isomorphic to the {\it normalizer} of transformations $T(n, \cdot), n \in \ZZ $. Since the centralizer is a normal subgroup of the normalizer, the study of automorphisms provides informations about the outer automorphism group of full groups.

 
 \noindent 
 In this paper we study realization properties of the centralizer of  aperiodic Cantor minimal system $(X,T,\Gamma)$, i.e,  when $X$ is a Cantor set and the action of the group $\Gamma$ on $X$ is  continuous, free and minimal (any countable group $\Gamma$  admits such an action, see   \cite{HjorthMolberg} or \cite{ABT} for an expansive action).

 \noindent 
 After recalling  necessary background on automorphisms and Cantor minimal systems that we call {\it generalized subshift}, we prove in Section \ref{sec:RealCountableGroup} that any countable group may appear in the normalizer of a  Cantor minimal aperiodic action of a free abelian group. The main result of Section \ref{sec:blackbox} shows that the centralizer of any aperiodic  Cantor minimal system given by the action of a countable group is a subgroup of the centralizer of a  $\ZZ$ minimal  Cantor system.
 
 \noindent 
We prove it is possible to realize  any residually finite group, possibly  infinitely generated, as a subgroup of the centralizer of a $\ZZ$ Cantor minimal system (Proposition \ref{cor:realizationResiduallyFiniteGroup}). 
Recently and independently of our work, Glasner, Tsankov, Weiss and Zucker extend this result by proving that any countable subgroup $G$ of a compact topological group embeds into the centralizer of an aperiodic minimal Cantor system \cite[Theorem 11.5]{GTWZ}.  


 \section{Definitions and background.}
 We say that $(X, T, \Gamma)$ is a {\em Cantor system} if $T:\Gamma\times X\to X$ is a continuous action on the Cantor set $X$.  For every $\gamma\in\Gamma$, we let $T^{\gamma}:X\to X$ denote the homeomorphism given by $T^{\gamma}(x)=T(\gamma, x)$, for every $x\in X$. The action is said {\em faithful} when the map $\gamma \mapsto T^{\gamma}$ is injective. We say that the Cantor system is {\em aperiodic} if the action $T$ is free, i.e, $T^{\gamma}(x)=x$ implies $\gamma=1_{\Gamma}$ for any $x\in X$. 
The Cantor system is {\em minimal} if for every $x\in X$, its orbit $o_T(x)=\{T^{\gamma}(x):\gamma\in\Gamma\}$  is dense in $X$. 
  The group generated by a collection of homeomorphisms $\{T_{i}\}_{i \in I}$ is denoted  $\langle T_{i} : i \in I \rangle$ and the one generated by the homeomorphisms of a $T$ action is simply denoted $ \langle T \rangle$.
%



 For a group $G$, we let  $\Aut(G)$  denote the group of automorphisms of $G$.  If $\Gamma$ is another group, $\Gamma\leq G$ means that $\Gamma$ is a subgroup of $G$ or isomorphic to a subgroup of $G$.  Recall that a (semi-)group $\Gamma$ is {\em residually finite} if for any $\gamma \neq \gamma' \in \Gamma$ there exists a homomorphism $\pi$ from $\Gamma$  to a finite (semi-)group $H$ such that $\pi(\gamma) \neq \pi(\gamma')$. A result of Mal'cev ensures that any finitely generated subgroup of $\GL(k ,\CC)$ is residually finite.

 \subsection{Group of automorphisms.}
Let $(X,T,\Gamma)$ be an aperiodic Cantor system.  The {\em normalizer} group of $(X, T, \Gamma)$, denoted $\Norm(T,\Gamma)$, is defined as the subgroup of all self-homeomorphism  $h$ of $X$  such that  $h \langle T \rangle h^{-1} = \langle T\rangle$, or equivalently,  there exists 
$\alpha_{h}\in \Aut(\Gamma)$  such that  $h\circ T^g=T^{\alpha_{h}(g)}\circ h$,  for every  $g\in \Gamma$.
 
By the aperiodicity of the action,  for any element $h \in \Norm(T, \Gamma)$ the  associated  automorphism $\alpha_{h} \in \Aut(\Gamma)$  is unique. Thus we can define
 $$\Aut(T, \Gamma)=\{ h \in \Norm(T, \Gamma):   \alpha_h=id\}.$$

It is direct to check that $\Aut(T, \Gamma)$ is a normal subgroup of $\Norm(T, \Gamma)$ and is the set of automorphisms of $(X,T, \Gamma)$.  More precisely, we have the following lemma. 
 \begin{lemma} \label{lem:AutoPlus}
For an aperiodic minimal Cantor system $(X,T, \Gamma)$ we have the following exact sequence 
 $$
\xymatrix{
\{1\} \ar[r] & \Aut(T, \Gamma)   \ar[r]^-{{\rm Id}} & {\Norm}(T, \Gamma) \ar[r]^-\alpha  & \Aut(\Gamma),
} 
$$
where $\alpha$ is the map 
$h \mapsto \alpha_{h}$.  
\end{lemma}
 \begin{proof}
It is enough  and straightforward to check that the  aperiodicity of the action implies that the map $\alpha$ is a group morphism. 
 \end{proof} 
It follows that the quotient group  $\Norm(T, \Gamma)/\Aut(T, \Gamma)$ is isomorphic to a subgroup of $\Aut(\Gamma)$. 
Thus, since $\Aut(\ZZ)$ is isomorphic to $\ZZ/2\ZZ$, for  a minimal Cantor system $(X, T, \ZZ)$, the group $\Aut(T, \ZZ)$ is a subgroup of  $\Norm(T, \ZZ)$ of index at most two.

\begin{lemma}\label{lem:freeness}
Let $(X,T, \Gamma)$ be a minimal system. Then the natural action of $\Aut (T, \Gamma)$ on $X$ is free.
\end{lemma}
\begin{proof} It is enough to observe that for an element $\phi \in \Aut (T, \Gamma)$, its set of fixed points  is a closed $T$ -invariant subset of $X$. By minimality, if not empty, this set is all $X$ and $\phi$ is the identity.    
\end{proof}
We will use this lemma to identify the automorphisms of a given minimal action. 

\subsection{Generalized subshifts.}\label{generalized-subshift}

We introduce  some notations and  system coming  from symbolic dynamics, we will use several times in this paper. 
 An {\em alphabet} $\A$  is a compact  (not necessarily finite) space with a metric ${\rm dist}_{\A}$.
 A {\em word} of  {\em length} $\ell$ is a  sequence $x_{1}\ldots x_{\ell}$ of $\ell$ letters in $\A$, and the set of such words is denoted $\A^\ell$.   The length of a word $u$ is denoted by $|u|$. 
 Any word can be interpreted as an element of the free monoid $\A^*$ endowed with the operation of concatenation.   For a  word $w =u.v$ that is the concatenation of two words $u$ and $v$, the words $u$ and $v$ are respectively  a {\em prefix} and a {\em suffix} of $w$.   
For an integer $\ell \in \NN$ and a word $u$,   the concatenation of $\ell$ times the word $u$ is denoted $u^{\ell}$ and $u^{\omega}$ is the bi infinite sequence  $(x_{n})_{n \in \ZZ}$ such that $x_{k|u|}\ldots x_{(k+1)|u| -1} = u$ for each integer $k \in \ZZ$. 
The set $\A^\ZZ$ is the collection of two sided infinite sequences $(x_{n})_{n\in \ZZ}$. 
 This last set is a compact  space for the product topology endowed with a metric
 $${\rm Dist}((x_{n})_{n}, (y_{n})_{n}) := \sum_{n \in \ZZ} 2^{-|n|} {\rm dist_{\A}}(x_{n}, y_{n}).$$ Note that if $\A$ is a Cantor set then, $\A^{\mathbb{Z}}$  also is a Cantor set. 

Let $\A^{+}=\bigcup_{n \geq 1} \A^n$ denote the set of all words on $\A$. 
A pseudometric ${\rm dist}_{\A^+}$ on $\A^{+}$  is  defined by ${\rm dist}_{\A^+} (u,v) = \max \{ {\rm dist}_\A (u_i , v_i) : i\leq \min (|u| , |v| ) \}$.

For a sequence ${\bf x}= (x_n)_{n}$, possibly infinite in $\A^{*} \cup \A^{\mathbb{Z}}$, we will use the notation ${\bf x}[i,j]$ to denote the word 
$x_i x_{i+1} \ldots x_{j}$ belonging to $\A^{j-i+1}$. 
We say that the index $i$ is an {\em occurrence} of  the word ${\bf x}[i,j]$  in ${\bf x}$. 
For an integer $\ell>0$, let $\L_{\ell}({\bf x})$ denote the set of words of length $\ell$, $\{{\bf x}[i,i+\ell-1], i \in\ZZ \}$ and  set the {\em language}  of $\bf x$ to  be the collection of words $\L({\bf x}) := \cup_{\ell >0 } \L_{\ell}({\bf x})$.

  We let $\sigma$ denote the {\em shift} map, that is the self-homeomorphism of $\A^\ZZ$ such that 
 $\sigma( (x_{n})_{n\in \ZZ}) = (x_{n+1})_{n\in \ZZ}$.   A {\em generalized subshift} is a topological dynamical system $(X, \sigma)$ where $X$ is a closed $\sigma$-invariant   subset of $\A^\ZZ$.
 The {\em language} of $X$ is the set $\L(X) := \cup_{\ell>0} \L_{\ell}(X)$ where $\L_{\ell}(X) := \cup _{{\bf x}\in X } \L_{\ell}({\bf x)}$ is the set of words of length $\ell$ occurring in some element of $X$.

\section{Realization of countable groups as subgroups of a normalizer}\label{sec:RealCountableGroup}

We  show in this section that any countable group can be realized as a subgroup of the centralizer of a minimal aperiodic system given by the action of a countable free abelian group  on the Cantor set. 

\begin{lemma}\label{cyclic-infinite}Let $\Gamma$ be a countable group. There exist an aperiodic Cantor system $(X,T,\Gamma)$ and  $f\in\Aut(T, \Gamma)$  such that $(X,f)$ is also aperiodic.  \end{lemma}
\begin{proof}
Let $\Gamma$ be a countable group. Since $\Gamma\oplus \ZZ$ is still a countable group, there exists an aperiodic Cantor system $(X, \phi, \Gamma\oplus \ZZ)$ (see  \cite{ABT, HjorthMolberg}).

 Let  $f:X\to X$ be the homeomorphism induced by the action of $(1_{\Gamma},1)\in \Gamma\oplus \ZZ$ on $X$, i.e,  $f=\phi^{(1_{\Gamma},1)}$. For every $g\in \Gamma$, let $T^g:X\to X$ be the homeomorphism induced by the action of $(g,0)$ on $X$, i.e, $T^g=\phi^{(g,0)}$. The new Cantor system $(X,T,\Gamma)$ is also aperiodic, and since $(1_{\Gamma},1)$ is in the center of $\Gamma\oplus\ZZ$, we have $f\in\Aut(T,\Gamma)$. Furthermore, since $\phi$ is aperiodic, we get  $f^n(x)=x$ implies $n=0$. 
\end{proof} 
     
\begin{proposition}
Let $G$ be a countable group.  Then there exist a countable subgroup  $\Gamma \leq\bigoplus_{\NN} \ZZ$ (a countable free abelian group)   and an aperiodic minimal Cantor system $(X,S,\Gamma)$, such that  $G\leq  \Norm(S,\Gamma)$.
 \end{proposition}

\begin{proof}
Let   $(X, T, G)$ be a Cantor aperiodic system and $f\in\Aut(T,G)$ as in Lemma \ref{cyclic-infinite}.  
Since $(X,f)$ is aperiodic, if $Y\subseteq X$ is a minimal component of $(X,f)$ then  $Y$ is a Cantor set. Observe that $T^g(Y)$ is also a minimal component of $(X,f)$, for every $g\in G$.  Consider the group ${\rm Stab}_G(Y)=\{g\in G: T^g(Y)=Y\}$
 and a collection $\{g_i: i\in I\}$ of elements of $G$ containing one and only one representative element of each class in the set of left cosets $G/{\rm Stab}_G(Y)$. 
We set
$$
\tilde{Y}=\prod_{i\in I}T^{g_i}(Y). 
$$
With the product topology $\tilde{Y}$ is a Cantor set.  

Set $\Gamma$ to be the group $\bigoplus_{I} \ZZ$, namely 
$$
\Gamma=\{(n_i)_{i\in I}\in \ZZ^{I}: n_i=0,  \textrm{ for all but a finite number of index } i\in I\}.
$$
Notice that $\Gamma$ is a countable free abelian group (if ${\rm Stab}_G(Y)$ is of finite index in $G$ then $\Gamma$ is finitely generated).

Given   $n=(n_i)_{i\in I}\in \Gamma$ and $y=(y_i)_{i\in I} \in \tilde{Y}$, we define
$$
S^{n}(y)=(f^{n_i}(y_i))_{i\in I}.
$$
Since each $T^{g_i}(Y)$ is invariant by $f$, we have that $S^n: \tilde{Y}\to \tilde{Y}$ is well defined and  is a homeomorphism. We call $S$ the action of $\Gamma$ on $\tilde{Y}$ induced by the $S^n$ maps.  It is  straightforward to show that  $(\tilde{Y}, S, \Gamma)$ is  aperiodic.  Since every $T^{g_i}(Y)$ is a minimal component of $(X,f)$, the system $(\tilde{Y}, S, \Gamma)$ is also minimal.

For every, $g\in G$ let define $\sigma_g: I\to I$ such that $gg_{\sigma_g(i)}\in g_{i}{\rm Stab}_G(Y)$.
We have that $\sigma_g$ is a permutation. Moreover, $\sigma_g$ induces the isomorphism $\alpha_g:\Gamma\to \Gamma$ given by 
$$
\alpha_g((n_i)_{i\in I})=(n_{\sigma_g(i)})_{i\in I}.
$$
 
For $g\in G$ we set 
$$
\tilde{T}^g(y)=(\tilde{y}_i)_{i\in I}, 
$$
where 
$$
\tilde{y}_i= T^{g}(y_{\sigma_{g}(i)}) \mbox{ for every } i\in I. 
$$
Since $T^{g}(y_{\sigma_{g}(i)})\in T^{g_i}(Y)$, we have that $\tilde{T}^g:\tilde{Y}\to \tilde{Y}$ is well defined.

 We have
\begin{eqnarray*}
\tilde{T}^{g}\circ S^n((y_i)_{i\in I}) & = & \tilde{T}^g((f^{n_i}(y_{i}))_{i\in I})\\
        & = & (T^g(f^{n_{\sigma_g(i)}}(y_{\sigma_g(i)})))_{i\in I} \\
        & = & (f^{n_{\sigma_g(i)}}(T^g(y_{\sigma_g(i)})))_{i\in I}    \\
\end{eqnarray*}

On the other hand,
\begin{eqnarray*}
  S^{\alpha_g(n)}\circ \tilde{T}^g((y_i)_{i\in I}) & = &  S^{\alpha_g(n)}\circ ((T^g( y_{\sigma_g(i)}  ) )_{i\in I} ) \\
          &=&  (f^{n_{\sigma_g(i)}}(T^g( y_{\sigma_g(i)}  ) ))_{i\in I}  \\
\end{eqnarray*}



This shows that $\tilde{T}^g\in \Norm(S,\Gamma)$, and since $g\mapsto \tilde{T}^g$ is an injective homomorphism,  we get $G\leq \Norm(S, \Gamma)$.
\end{proof}

\section{Subgroups of the centralizer of a $\ZZ$ minimal Cantor system}\label{sec:blackbox}

We study the class of  subgroups of automorphisms of a $\ZZ$ Cantor minimal system. We start by showing that this class contains the centralizer of any aperiodic Cantor minimal system of a countable group. From this we deduce this class is stable  by direct product and contains any countable residually finite group.
 
\subsection{Realization of subgroups of the centralizer}\label{sec:realSubGroup}

\begin{proposition}\label{prop:injection}
Let $(X,T, \Gamma)$ be a  Cantor aperiodic  minimal system where $\Gamma$ is a countable group. 
Then there exists a Cantor minimal  system $(Y,S, \ZZ)$  such that  
$\Aut(T,\Gamma) \leq\Aut(S,\ZZ).$ 
\end{proposition}

To prove Proposition \ref{prop:injection}, we first define the Cantor minimal system. 
Using  the notion of generalized subshift  introduced in Section \ref{generalized-subshift},  we will consider a $S$-adic subshift over the infinite alphabet $X$ thanks the notion of generalized substitution  introduced in \cite{DurandOrmesPetite}.  
To show that the generated  system is a Cantor minimal system we will need to use similar results of \cite{DurandOrmesPetite} but adapted in the $S$-adic case. 
The proofs are just small modifications of the original arguments. 
We set them for completeness.

Let $\Gamma$ be a countable group and let  $\{\S_{n}\}_{n}$ be a nested sequence of finite and symmetric sets of elements $\S_{n} = \S_{n}^{-1} =\{s_{1}, \ldots, s_{d_{n}}\}$ such that the union $\bigcup_n\S_n$ generates the group $\Gamma$, where $s_{1}$ is the identity and $d_0\geq 1$. 
When $\Gamma$ is finitely generated, this sequence is stationary. 
Let $(X,T, \Gamma)$ be a  Cantor aperiodic  minimal system. 




Let us recall the definition of  {generalized substitution}   on an alphabet space $X$.
We need to deal with several topological considerations that are trivial in the case where the alphabet is finite. 
For a word $w \in X^*$ and $1 \leq j \leq |w|$, let $\pi_j(w)$ denote the $j$th letter of $w$. 
We say that $\tau:X \to X^*$ is a {\em generalized substitution on $X$} if 
$a \mapsto |\tau(a)|$ is continuous and 
the projection map $\pi_j \circ \tau$ is continuous on the set $\{a \in X : |\tau(a)| \geq j\}$. 
The words $\tau (z)$, $z\in X$, are called $\tau${\em-words}.
A substitution $\tau $ can be extended by concatenation on all the words on $X$ and on all the sequences in $X^{\ZZ}$ by a map still denoted $\tau$. More precisely,  for $x=(x_i)_{i\in\ZZ}\in X^{\ZZ}$,  
$$
\tau(x)=\cdots \tau(x_{-1}). \tau(x_0)\tau(x_1)\cdots.
$$
Consider, for each integer $n \ge 0$  the substitution $\tau_{n}$ on $X$ defined by 
$$\tau_{n} \colon x \mapsto T^{s_{1}}(x)T^{s_{2}}(x) \ldots T^{s_{d_{n}}}(x)$$ where  we recall that $\{s_{1}, \ldots, s_{d_{n}}\} = \S_{n}$. 
When $\Gamma$ is finitely generated, the sequence $(\tau_{n})_{n}$ is stationary. 

Set $ \bar{\tau}$ to be the sequence of substitutions  $(\tau_{n})_{n}$.
We shall consider $\mathcal{L}(\bar{\tau})$ 
the {\em language generated by} $\bar{\tau}$, again a trickier notion to define than in the classical case. 
Fix a letter $a$ in the alphabet space $X$. 
By the {\em language generated by $a$}, 
denoted $\mathcal{L}(\bar{\tau} , a)$, we mean the set of words $w \in X^*$ such that $w$ is a subword of $\tau_{0}\circ \cdots \circ \tau_{j} (a)$ for some $j \in \mathbb{N}$, or $w$ is the limit in $X^n$ of such words. 
We set $\mathcal{L}(\bar{\tau}) = \cup_a \mathcal{L}(\bar{\tau}, a)$.
We also  define  $X_{\bar\tau} \subset X^{\mathbb{Z}}$ to be set of sequences ${\bf x} \in X^{\mathbb{Z}}$ such that
${\bf x}[-n,n] \in \mathcal{L}(\bar{\tau})$ for all $n \geq 0$.
It follows that $X_{\bar\tau}$ is a generalized subshift, i.e., a closed, $\sigma$-invariant subset of $X^{\mathbb{Z}}$.

As in the classical case, our generalized $S$-adic system $\bar \tau$ is \emph{primitive} in the following sense. 

\begin{lemma}\label{lem:primitive}
Given any non-empty open set $V \subset K$, there is an $n \in \mathbb{N}$ such that 
for any letter $a \in X$ and any $k \geq n$ and any $i \ge 0$, one of the letters of $\tau_{i}\circ \cdots \circ \tau_{k+i}(a)$ is in the set $V$.
\end{lemma}
\begin{proof}
From a well known result of Auslander, for the minimal $\Gamma$-action there exists a finite set $K \subset \Gamma$ such that  for any $x \in X$, the set of ``return times'' $R_{V}(x) :=\{ g \in \Gamma : T^{g}(x) \in V\}$ is $K$-syndetic, meaning that $\Gamma =  KR_{V}(x) = \{ kg: k \in K,  g \in R_{V}(x)\}$.

Also observe that for any finite set $K \subset \Gamma$, there exists an integer $n$ such that for any $x\in X$, the letters $T^{g}(x)$, $g \in K^{-1}$ occur in $\tau_{0}\circ \cdots \circ \tau_n(x)$. 
Since the sets $\S_{n}$ are nested, each word $\tau_{n}(x)$ is a prefix of $\tau_{m}(x)$ for $m\ge n$, and the letters $T^{g}(x)$, $g \in K^{-1}$ occur in $\tau_{j}\circ \cdots \circ \tau_{n+j}(x)$ for any $j\ge 0$.  
Theses two facts imply the  primitivity of $\bar \tau$. 
\end{proof}

As in the case of primitive generalized substitution \cite[Proposition 19]{DurandOrmesPetite},  the assumption of primitivity  simplifies the definition of the language.
\begin{lemma}\label{lem:langage}
 For any two letters $a, b \in X$, $\mathcal{L}({\bar \tau} , a) = \mathcal{L}({\bar \tau} , b)$. 
\end{lemma}
This enable us to denote  $\mathcal{L}({\bar \tau} ):= \mathcal{L}({\bar \tau} , a)$ independently of the letter $a \in X$. 

\begin{proof}
Let $w \in \mathcal{L}( {\bar \tau}, b)$ and suppose $\epsilon>0$ is given. Then, there are integers $k,i,j \geq 0$ such that
the distance between $\tau_{0}\circ \cdots \circ \tau_k(b)[i,j]$ and $w$ is less than $\epsilon/2$ in the $X^{|w|}$-metric.
By the continuity  of the action $T$, there is a $\delta>0$ such that 
${\rm dist}_{X}(b,b')<\delta$ implies the distance from $\tau_{0}\circ \cdots \circ \tau_k(b)[i,j]$ and $\tau_{0}\circ \cdots \circ \tau_k(b')[i,j]$ is less than $\epsilon/2$.

By the primitivity of $\bar \tau$, there is an integer  $n$ and a $b'$ at distance smaller than $\delta$ of $b$  such that $b'$ occurs in $\tau_{k+1}\circ \cdots \circ \tau_{k+n+1}(a)$. 
It follows that there are integers $i', j' \geq 0$ such that 
the distance from $\tau_{0}\circ \cdots \circ \tau_{k+n+1}(a)[i',j']$ to $w$ is less than $\epsilon$. 
Thus, $w $ belongs to $ \mathcal{L}({\bar \tau} , a)$ and
$\mathcal{L}({\bar \tau} , b)$ is included in  $\mathcal{L}({\bar \tau} , a)$. 

Similarly, $\mathcal{L}({\bar \tau} , a) \subset \mathcal{L}({\bar \tau} , b)$. 
\end{proof}


Th next lemma shows the subshift  $X_{\bar\tau}$ is not empty and provides a dense orbit. 
\begin{lemma}
\label{prop:minimality}
Let $a \in X$ and let ${\bf x}$ be any point in $X^{\mathbb{Z}}$ with $x_{-1}=T^{s_{1}}(a) \in X$ and $x_0=T^{s_{2}}(a)\in X$.
Then,  for any accumulation point ${\bf z}$ of  the sequence $(\tau_{0}\circ\cdots\circ\tau_{n} ({\bf x}))_{n}$, $X_{\bar\tau}$ is the closure of the $\sigma$-orbit of ${\bf z}$. 
\end{lemma}
\begin{proof}
Suppose  ${\bf z}$ is the limit of  sequences $\tau_{0}\circ \cdots \circ \tau_{k_i} ({\bf x})$ for some increasing integer sequence  $(k_{i})_{i}$. 
Fix $m \in \mathbb{N}$ and let $\epsilon >0$ be given. Then, there is an $i \geq 1$ such that 
the word ${\bf  z}[-m,m]$ is within $\epsilon$ of a subword of $\tau_{0}\circ \cdots \circ \tau_{k_i}(x_{-1}x_{0})$.
Since $x_{-1}x_{0}$ is a prefix of $\tau_{k_{i}+1}(a)$,   ${\bf z}[-m,m]$ is 
within $\epsilon$ of a subword of $\tau_{0}\circ \cdots \circ \tau_{k_i+1}(a)$. 
Therefore, ${\bf z}[-m,m] \in \mathcal{L}({\bar \tau} , a) \subset \mathcal{L}({\bar \tau})$. 
This shows that ${\bf z}$ belongs to $X_{\bar \tau}$.
But because $X_{\bar \tau}$ is closed and shift-invariant, 
the orbit closure $\overline{\{\sigma^n ({\bf  z}) : n\in \mathbb{Z} \}}$ is included in  $X_{\bar \tau}$.

Conversely, fix an ${\bf y}\in X_{\bar\tau}$ and an  $\epsilon'>0$.  Let us fix an integer  $n \ge 0$, such that   ${\rm Dist}({\bf y},{\bf y'})< \epsilon'$ when  ${\rm dist}_{X^{+}}( {\bf y}[-n,n], {\bf y'}[-n,n])$ is less than $\epsilon'/2$. 
Since ${\bf y}[-n,n] \in \mathcal{L}({\bar \tau}) =\mathcal{L}({\bar \tau}, x_{0}) $, there are integers $k,i$  such that 
$\dist_{X^{+}}(\tau_{0}\circ \cdots \circ \tau_k (x_{0})[i-n,i+n],{\bf y}[-n,n]) $ is less than $\epsilon'/4$. 

By continuity, there exists a neighborhood $U$ of $x_{0}$  such that  the words $\tau_{0}\circ \cdots \circ \tau_k(x_{0})$ and $\tau_{0}\circ \cdots \circ \tau_k(x')$ are within $\epsilon'/4$ whenever $x'\in U$.  

The primitivity of $\bar\tau$ implies there is a $\ell $ such that  for any  letter $b\in X$,  a point in $U$  occurs in the word $\tau_{k+1}\circ \cdots \circ \tau_{k+\ell}(b)$. 
Set $R= | \tau_{0}\circ \cdots \circ \tau_{k+\ell}(b)|$ for some (all) $b \in X$.
Using the compacity of the space of letters $X$,  it is standard to check the word ${\bf z}[ R-1, 2R-1]$ is of the form $\tau_{0}\circ \cdots \circ \tau_{k+\ell}(b)$ for some $b \in X$.
Finally,  we get the word ${\bf y}[-n,n]$ is within $\epsilon'/2$ of a subword of ${\bf z}[ R-1, 2R-1]$.
Since $\epsilon'$ is  arbitrary,  it follows that ${\bf y} \in \overline{ \{ S^n ({\bf z}) : n \in \mathbb{Z} \} }$.
\end{proof}

\begin{lemma}
\label{lem:minimality}
The system $(X_{\bar\tau}, \sigma)$ is a minimal Cantor system.
\end{lemma}

\begin{proof}
Let ${\bf x}\in X^\mathbb{Z}$ where $x_{-1}x_0$ is a prefix of $\tau_{0}(a)$ for a fixed $a \in X$. 
Let ${\bf z}$ be an  accumulation point of  the sequence $(\tau_{0}\circ\cdots\circ\tau_{n} ({\bf x}))_{n}$.
Thus, from Lemma \ref{prop:minimality}, the $\sigma$-orbit closure of ${\bf z}$ is $X_{\bar\tau}$.

To show the minimality it is enough to prove that for every open set $U$ in $X_{\bar\tau}$ containing ${\bf z}$,
there is an $R>0$ such that for any $j \in \mathbb{Z}$, there is a $0 \leq i < R$ such that $\sigma^{j+i}({\bf z}) \in U$. 
Notice that there are  $n \in \mathbb{N}$ and  $\epsilon>0$ such that $\sigma^i({\bf z})$  is in $U$ whenever
the distance from ${\bf z}[-n,n]$ to ${\bf z}[i-n,i+n]$ is less than $\epsilon$. 

Moreover, from the continuity of the action $T$ and the very definition of ${\bf z}$, there are a $\delta>0$ and a $k>0$ such that ${\rm dist}_X(a,a')<\delta$ implies that a word 
occurs in $\tau_{0}\circ \cdots \circ \tau_k(a')$ which is within $\epsilon$ of ${\bf z}[-n,n]$. 

By the primitivity of $\bar\tau$, there is an $m$ such that for any $b \in X$, 
$\tau_{k+1}\circ \cdots \circ \tau_{m+k}(b)$ contains a letter within $\delta$ of $a$. Accordingly, any  $\tau_{0}\circ \cdots \circ \tau_{m+k}$ word
$\tau_{0}\circ \cdots \circ \tau_{m+k}(b)$ contains a word within $\epsilon$ of ${\bf z}[-n,n]$.

Now observe that ${\bf z} = \lim_{i \to \infty} \tau_{0}\circ \cdots \circ \tau_{k_i} ({\bf x})$. 
Let $\bf y$ be an accumulation point of the sequence $( \tau_{m+k+1}\circ \cdots \circ \tau_{k_i} ({\bf x}))_{i}$. 
The continuity of the substitutions implies  that ${\bf z} = \tau_{0}\circ \cdots \circ \tau_{m+k} ({\bf y})$ so that ${\bf z}$ is a concatenation of  $\tau_{0}\circ \cdots \circ \tau_{m+k}$-words.   
Since all $\tau_{0}\circ \cdots \circ \tau_{m+k}$-words have the same length, say $R$ , it follows from the  first part of the proof that
for any $j \in \mathbb{Z}$, there is a $0 \leq i < R$ such that $\sigma^{j+i}({\bf z}) \in U$.

To show the aperiodicity of the system,  notice that the language  ${\mathcal L}({\bf z}) \subset  \mathcal{L}(\tau)$ contains infinitely many letters. 
This together with the minimality of the system imply the aperiodicity.
\end{proof}


\begin{proof}[Proof of Proposition \ref{prop:injection}]
Let $(X_{\bar\tau},\sigma)$ be the generalized subshift defined above.  
Let $\phi$ be in ${\Aut}(T,\Gamma)$, that is  $\phi \colon X \to X$ continuous and $\phi\circ T^{\gamma}(x) = T^{\gamma}\circ\phi(x)$ for all $x\in X$, $\gamma \in \Gamma$. We can associate to $\phi$ a transformation $\bar{\phi}$ on $X^{n}$, by coordinate wise composition 
\begin{align}\label{eq:barphi1}
 \bar{\phi} ((x_{i})_{i}) := (\phi(x_{i}))_{i}.
 \end{align}
 
This defines by concatenation a continuous bijective map on the whole space $X^{\ZZ}$,  still denoted $\bar{\phi}$, that commutes with the shift. 
Observe moreover that, for every $x\in X$, and every integer $n$ 
\begin{align}\label{eq:barphi2}
\tau_{n}(\phi(x)) = \bar{\phi}(\tau_{n}(x)).
\end{align}
It follows that the map $\bar{\phi}$ preserves the subshift $X_{\bar\tau}$.   
It is then straightforward to check, with Lemma \ref{lem:freeness}, that  the map $\phi \in {\Aut}(T,\Gamma) \mapsto \bar{\phi} \in {\Aut}(\sigma,\ZZ)$  is an injective homomorphism.
\end{proof}

{\rm A first consequence of Proposition \ref{prop:injection} is that the automorphisms group of a Cantor minimal $\ZZ$ system may be uncountable.
}

\begin{coro}\label{cor:realisationProFiniteGroup}
Let $\bf G$ be a topological group homeomorphic to a Cantor set.
Then there exists a Cantor minimal $\ZZ$ system $(X,S, \ZZ)$ such that ${\bf G} \leq\Aut (S,\ZZ)$. 
\end{coro}
\begin{proof}
The group ${\bf G}$ is separable so the (countable) group $\Gamma \le {\bf G}$ generated by a dense countable subset is dense in $\bf G$.
For every $\gamma\in \Gamma$, and $g\in {\bf G}$, we set $T^{\gamma}(g)=\gamma g$. Thus $({\bf G}, T, \Gamma)$ is the minimal aperiodic Cantor system induced by the left translations of $\Gamma$ on ${\bf G}$.
The group $\bf G$ acts  by right translations  on itself.
This action is transitive and commutes with the one of $\Gamma$ so by Lemma \ref{lem:freeness}, $\Aut(T,\Gamma)$ is isomorphic to ${\bf G}$.
Finally Proposition \ref{prop:injection} provides the result.  
\end{proof}

If $\Gamma$ is a finitely generated residually finite group, it is  isomorphic to a dense subgroup of any  $\Gamma$-odometers (see \cite{CP08}), which is a topological group homeomorphic to the Cantor set.  The application of Corollary  \ref{cor:realisationProFiniteGroup}  to a $\Gamma$-odometer gives a $\ZZ$ Cantor minimal system with $\Gamma$ a subgroup of its centralizer. 
Actually the same result is true for countable residually finite group that might be  infinitely generated and with non-equicontinuous action. 
We provide  here a direct proof based on  a Lindenstrauss-Weiss idea \cite[Proposition 3.5]{LW}. 

\begin{proposition}\label{cor:realizationResiduallyFiniteGroup}
Let $\Gamma$ be a  countable residually finite group. Then there exists a Cantor minimal  system $(X,S, \ZZ)$ such that $\Gamma \leq\Aut(S,\ZZ)$.
\end{proposition}
\begin{proof}
Recall that a countable residually finite group $\Gamma$ always admits a faithful action on the Cantor set $X$ where the set of points  with a finite orbit is dense \cite[Theorem 2.7.1]{CSC}.  We will construct a generalized subshift on the alphabet $X$ (see Section \ref{generalized-subshift}). 

For each $n \ge 1$ and $\gamma\in \Gamma$  let $\gamma_{n} \colon X^n \to X^n$ be the homeomorphism defined as $\gamma_{n}(x_{1}, \ldots,x_{n}) = (\gamma(x_{1}), \ldots, \gamma(x_{n}))$, and let $\gamma_{\omega} \colon X^\ZZ \to X^\ZZ$  be the homeomorphism defined as $\gamma_{\omega}((x_{n})_{n}) = (\gamma(x_{n}))_{n}$. This provides  $\Gamma$-actions on $X^{n}$ and $X^{\ZZ}$.

It follows that the  shift map $\sigma \colon X^\ZZ \to X^\ZZ$ and $\gamma_{\omega}$ commute.  We have to construct now the specific  minimal generalized subshift.

Set  $B_{1}$  as the collection of all the words of length $1$, in the alphabet $X$.
Fix $n\ge 2$ and suppose that at the step $n-1$ we have defined a collection of words $B_{n-1}$ such that:
\begin{itemize}
  \item All of its elements have the same length $\ell_{n-1}$.
  \item The set $B_{n-1}$  is closed on $X^{\ell_{n-1}}$ and  preserved by  each map $\gamma_{\ell_{n-1}}$.
  \item The set of  points with a finite $\Gamma$-orbit is dense in $B_{n-1}$.
\end{itemize}
Let $\{x_{1,n}, \dots, x_{k_{n},n}\} \subset B_{n-1}B_{n-1} $ be a finite   $\Gamma$-invariant set that is $1/n$-dense in $B_{{n-1}}B_{n-1} \subset X^{2\ell_{n-1}}$. Of course,  each $x_{i,n}$ has a finite $\Gamma$-orbit. 
Set  $B_{n}$ to be the collection of all the words $w$ that are a concatenation of words in $B_{n-1}$ of  the form 
$$w =  w_{1}\ldots w_{n}x_{s(1),n} \ldots  x_{s(k_{n}),n},$$   
where $w_{1}, \ldots, w_{n}$ are words in $B_{n-1}$ and $s$ is a permutation of $\{1, \ldots, k_{n}\}$. It is clear that each word in $B_{n}$ is of length $(n+2k_{n})\ell_{n-1} =: \ell_{n}$, each  map $\gamma_{\ell_{n}}$, with $\gamma \in \Gamma$, preserves the   set  $B_{n}$ and that the  points with a finite $\Gamma$-orbit are dense in  the closed set $B_{n} \subset X^{\ell_{n-1}}$.

Let $X_{n}$ be the subshift whose any element is a concatenation of words in $B_{n}$ and let $X_{\infty}$ be the subshift $X_{\infty} = \bigcap_{n\ge 1} X_{n}$. The minimality follows from the next claim. 

\medskip

{\bf {\underline{Claim.}}}\label{cl:minimality}
For each $n\ge 1$, for any $x \in X_{n+1}$ and  $y \in X_{n}$, there exists an integer $ | \ell | \le \ell_{n}  $ such that ${\rm Dist}(\sigma^\ell (x), y) \le  1/n + 3/2^{\ell_{n-1}}$.

   For two words $w_{1}, w_{2} \in B_{n}$, there exists a $x_{i,n+1}$ such that the word $w_{1}w_{2}$ is at distance less than $1/(n+1)$ from $x_{i,n+1}$ in $X^{2\ell_{n} }$.  
Since $y$ is a concatenation of words in $B_{n}$ and since any word $x_{i,n+1}$ appears in each word of $ B_{n+1}$, a direct computation provides the claim.

Clearly $X_{\infty}$ is invariant by every $\gamma_{w}$.  
It remains to check that $\Gamma$ acts faithfully on $X_{\infty}$.  Let $\gamma \in \Gamma$  be such that $\gamma_{w}$ has a fixed point $x \in X_{\infty}$, then $\gamma$ fixes each letter of $x$. A direct induction proves that the set of letters occurring in the words $x_{i,n}$, $n \ge 1$ is dense in $X$. Hence in particular, $\gamma$ fixes a dense set of points in $X$, so $\gamma$ is the identity.
\end{proof}

Of independent interest, we also deduce that the property ``to be a subgroup of automorphism of an aperiodic minimal  Cantor system'' is stable by direct product.
\begin{coro}\label{cor:StableWreathCartesianProduct}
Let $(X,T,G)$ and $(X,S,H)$ be two aperiodic Cantor minimal systems for two countable groups $G$ and $H$. Then there exists a $\ZZ$ Cantor minimal system $(X,R,\ZZ)$ such that ${\Aut}(T,G)\oplus {\Aut}(S,H) \le {\Aut}(R,\ZZ)$. 
\end{coro}
%
\begin{proof}
 Observe that the direct product  ${\Aut}(T,G)\oplus  {\Aut}(S,H)$ is a subgroup of the centralizer of  the  product system on the space $X \times X$ for the product action of $G \oplus H$. This action is aperiodic and minimal so  a direct application of Proposition \ref{prop:injection} gives the conclusion. 
\end{proof}

{\it Acknowledgment.} The authors thank an anonymous referee for his careful reading and pointing out mistakes in the initial version of this paper.

\end{document}